\documentclass[copyright,creativecommons]{eptcs}
\providecommand{\event}{GandALF 2014} 
\usepackage{breakurl}             
\usepackage{amsfonts}
\usepackage{amssymb}
\usepackage{amsmath}
\usepackage{amsthm}
\usepackage{verbatim}

\newtheorem{theorem}{Theorem}[section]
\newtheorem{lemma}[theorem]{Lemma}

\title{Some Turing-Complete Extensions of First-Order Logic}
\author{Antti Kuusisto\thanks{This work was carried out during a tenure of the
ERCIM ``Alain
Bensoussan" Fellowship Programme.
The reported research has received funding from the
European Union
Seventh Framework
Programme (FP7/2007-2013) under
grant agreement
number 246016; the author has
been employed as Marie Curie Fellow during the research period.}
\institute{Institute of Computer Science\\ University of Wroc\l aw\\ Poland}
\email{antti.j.kuusisto@uta.fi}
}
\def\titlerunning{Some Turing-Complete Extensions of First-Order Logic}
\def\authorrunning{A. Kuusisto}
\begin{document}
\maketitle
\begin{abstract}
We introduce a natural Turing-complete extension of first-order logic $\mathrm{FO}$.
The extension adds two novel features to $\mathrm{FO}$.
The first one of these is the capacity to \emph{add new points} to models
and \emph{new tuples} to relations. The second one is the possibility of
\emph{recursive looping} when a formula is evaluated using a semantic game.
We first define a game-theoretic semantics for the logic and
then prove that the expressive power of the logic
corresponds in a canonical way to the recognition
capacity of Turing machines.
Finally, we show how to incorporate generalized quantifiers into the logic
and argue for a highly natural connection between 
oracles and generalized quantifiers.
\end{abstract}

\section{Introduction}
We introduce a natural Turing-complete extension of
first-order logic $\mathrm{FO}$.
This extension essentially
adds two features to basic $\mathrm{FO}$.
The first one of these is the capacity to \emph{add new points} to models
and \emph{new tuples} to relations. The second one is the possibility of
\emph{looping} when a formula is evaluated using a semantic game.
Logics with different kinds of recursive looping capacities have been widely
studied in the context of finite model theory \cite{libkin}. Typically
such logics are fragments of second-order predicate logic.
A crucial weakness in the expressivity of $k$-th order predicate
logic is that only a finite amount of information can be encoded by a finite number of
quantified relations over a finite domain.
Intuitively, there is \emph{no infinitely expandable memory} available.
Thus $k$-th order logic is not Turing-complete.
To overcome this limitation, we add to first-order logic operators that
enable the addition of new elements to the domains of models and
new tuples to relations.
Coupling this feature with the possibility of recursive looping leads to a
very natural Turing-complete extension of first-order logic.
In addition to operators that enable the extension of domains and relations,
we also consider an operator that enables the deletion of tuples of relations.
It would be natural to also include to our framework an operator that enables the deletion of
domain points. This indeed could (and perhaps should) be done, but for purely
technical reasons, we ignore this possibility.
We provide a game-theoretic semantics for the novel logic.
A nice feature of the logic\hspace{0.4mm}---\hspace{0.4mm}let us simply call it $\mathcal{L}$,
or \emph{logic} $\mathcal{L}$---\hspace{0.4mm}is that it
\emph{simulates halting as well as diverging} computations
of Turing machines in a very natural way.
This behavioural correspondence between Turing machines
and the logic $\mathcal{L}$ stems from the
appropriate use of game-theoretic concepts.
Let us have a closer look at this matter.
Let $\mathfrak{A}$ be a model and
$\varphi$ a formula of first-order logic. Let $f$ be 
an \emph{assignment function} that interprets the
free variables of $\varphi$ in the domain $A$
of $\mathfrak{A}$. The semantic game $G(\mathfrak{A},f,\varphi)$
is played between the two players $\exists$ and $\forall$
in the usual way (see, e.g., \cite{hintikka22, mann}).
If the verifying player $\exists$ has a
winning strategy in the game $G(\mathfrak{A},f,\varphi)$,
we write $\mathfrak{A},f\models^+\varphi$
and say that $\varphi$ is \emph{true} in $(\mathfrak{A},f)$.
If, on the other hand, the falsifying player $\forall$ 
has a winning strategy, we write $\mathfrak{A},f\models^-\varphi$
and say that $\varphi$ is \emph{false} in $(\mathfrak{A},f)$.
Since $\varphi$ is a first-order formula, we have $\mathfrak{A},f\models^+\varphi$
iff it is not the case that $\mathfrak{A},f\models^-\varphi$.
If $\varphi$ is a formula of \emph{IF logic} \cite{hintikka} or \emph{dependence logic}
\cite{vaananen}\hspace{0.4mm},
for example, the situation changes. It is then possible that \emph{neither player
has a winning strategy} in the semantic game. This results in a third truth value (\emph{indeterminate}).
Turing machines of course exhibit analogous behaviour: on an input word $w$, a
Turing machine can halt in an accepting state, halt in a rejecting state, or diverge.
The logic $\mathcal{L}$ incorporates each of these three options
in its semantics in a canonical way.
For each Turing machine $\mathrm{TM}$, there exists a sentence of
$\varphi_{\mathrm{TM}}$ such that $\mathrm{TM}$ \emph{accepts} 
the encoding of a \emph{finite} model $\mathfrak{A}$ iff $\mathfrak{A},f\models^+\varphi_{\mathrm{TM}}$,
and furthermore, $\mathrm{TM}$ \emph{rejects} the encoding of $\mathfrak{A}$
iff $\mathfrak{A},f\models^-\varphi_{\mathrm{TM}}$.
Therefore $\mathrm{TM}$ \emph{diverges} on the encoding of $\mathfrak{A}$
iff neither the verifying nor the falsifying player has a winning strategy
in the game invoving $\mathfrak{A}$, $f$ and $\varphi$.
For the converse, for each formula $\chi$ of the logic $\mathcal{L}$, there exists a
Turing machine $\mathrm{TM}_{\chi}$ such that a similar full
symmetry exists between semantic games involving $\chi$ and
computations involving $\mathrm{TM}_{\chi}$.
By Turing-completeness of a logic we mean exactly
this kind of a behavioral equivalence between Turing machines and formulae.
The moves in the semantic games for $\mathcal{L}$
are exactly as in first-order logic in positions involving
the first-order operators $\exists x$, $\vee$, $\neg$.
In positions of the type $Ix\, \varphi$, a
fresh point is inserted into the domain of the model investigated,
and the variable $x$ is interpreted to refer to the fresh point.
There are similar operators for the insertion (deletion) of tuples into (from) relations.
The recursive looping is facilitated by operators
such as the ones in the formula $1\bigl(P(x)\vee\, 1)$,
where the player ending up in a position involving the
novel atomic formula $1$ can \emph{jump back} into a
position involving $1\bigl(P(x)\vee\, 1\bigr)$.
Semantic games are played for at most omega
rounds\hspace{0.6mm}\footnote{\emph{Omega} of course refers to the
smallest infinite ordinal.} and can be won only
by moving to a position involving a first-order atomic
formula. Winning and losing in positions involving 
first-order atoms is determined exactly as in first-order logic.
Operators that bear a resemblance to the ones
used in the logic $\mathcal{L}$ have of 
course been considered in logical
contexts before. Lauri Hella has suggested (personal communication)
extending first-order logic with recursive looping constructors
that resemble those investigated in this article. His idea involves studying fixed point logics
using a game-theoretic semantics with somewhat different
kinds of winning conditions than the ones we shall formulate below.
The framework does not involve modifying the domains of structures.
The insertion
(deletion) of tuples to (from) relations is
an important ingredient in dynamic
complexity (see, e.g., \cite{gra, imme}),
although motivated and used there in a way that is
quite different from the approach in this article.
Logics that involve jumping into different model
domains include for example 
the \emph{sort logic} of 
V\"{a}\"{a}n\"{a}nen (\cite{vaah},\cite{vaa}), a logic which
can in a sense be regarded as the strongest
possible model theoretic logic.
Other systems with similar constructors or motivations to those
considered in this article include for example BGS logic 
\cite{bgs},
\begin{footnotesize}
WHILE
\end{footnotesize}
languages \cite{abiteb}
and abstract state machines
\cite{guuure1, borger}.
See also the articles \cite{chandra} and \cite{japaridze}.
The reason we believe that the logic $\mathcal{L}$ is
particularly interesting lies in its simplicity and
\emph{exact behavioural
correspondence} with Turing machines on one hand,
and in the fact that it provides a 
\emph{canonical unified perspective} on logic and
computation on the other hand.
The logic $\mathcal{L}$ canonically
extends ordinary first-order logic
to a Turing-complete framework, and thereby serves not
only as a novel logic, but also as a novel model of computation.
It is also worth noting that the fresh operators of $\mathcal{L}$ nicely
capture two classes of constructors that are omnipresent in
the practise of mathematics: 
scenarios where fresh points are added to investigated constructions
(or fresh lines are drawn, etc.) play a central role in geometry, and
recursive looping operators are found everywhere in mathematical practise,
often indicated with the help of the famous three dots (...).

The structure of the paper is as follows.
In Section \ref{two} we define some
preliminary notions and give a formal
account of the syntax $\mathcal{L}$.
In Section \ref{sectionthree} we develop
the semantics of $\mathcal{L}$.
In Section \ref{four} we establish the Turing-completeness
of $\mathcal{L}$ in restriction to the class of word models.
In Section \ref{five} we use the results of
Section \ref{four} in order to establish Turing-completeness
of $\mathcal{L}$ in the class of all finite models.
In Section \ref{six} we show how to extend $\mathcal{L}$
with generalized quantifiers. We also briefly discuss the conceptiual
link between oracles and generalized quantifiers.
%

%

%
%

%
%

%
\section{Preliminaries}\label{two}
Let $\mathbb{Z}_+$ denote the set of positive integers.
Let $\mathrm{VAR}\, :=\, \{\, v_i\ |\ i\in\mathbb{Z}_+\ \}$
be the set of variable symbols used in first-order logic.
We mainly use \emph{metavariables} $x,y,z,x_i,y_i,z_i,$ etc., in order
to refer to the variables in $\mathrm{VAR}$.
Let $k\in\mathbb{Z}_+$. We let $\mathrm{VAR}_{\mathrm{SO}}(k)$
be a countably infinite set of $k$-ary relation variables. 
We let $\mathrm{VAR}_{\mathrm{SO}} = \bigcup_{k\in\mathbb{Z}_+}\mathrm{VAR}_{\mathrm{SO}}(k)$.
Let $\tau$ denote a complete relational vocabulary, i.e.,
$\tau$ is the union $\bigcup_{k\in\mathbb{Z}_+}\tau_k$, where
$\tau_k$ is a countably infinite set of $k$-ary relation symbols.
Let $\sigma\subseteq\tau$.
Define the language $\mathcal{L}^*(\sigma)$ to be the smallest set $S$ such
that the following conditions are satisfied.
\begin{enumerate}
\item
If $x_1,...,x_k$ are variable symbols and $R\in\sigma$ a 
$k$-ary relation symbol, then $R(x_1,...,x_k)\in S$.
\item
If $x_1,...,x_k$ are variable symbols and $X\in\mathrm{VAR}_{\mathrm{SO}}(k)$ a 
$k$-ary relation variable, then $X(x_1,...,x_k)\in S$.
\item
If $x,y$ are variable symbols, then $x = y\ \in\ S$.
\item
If $k\in\mathbb{N}$ is (a symbol representing) a natural number,
then $k\in S$.
\item
If $\varphi\in S$, then $\neg\varphi\in S$.
\item
If $\varphi,\psi\in S$, then $(\varphi\wedge\psi)\in S$.
\item
If $x$ is a variable symbol and $\varphi\in S$,
then $\exists x\, \varphi\ \in\ S$.
\item
If $x$ is a variable symbol and $\varphi\in S$,
then $Ix\, \varphi\ \in\ S$.
%
%
%
\item
If $x_1,...,x_k$ are variable symbols, $R\in\sigma$ a $k$-ary relation
symbol and $\varphi\in S$,
then $I_{Rx_1,...,x_k}\varphi\ \in\ S$.
\item
If $x_1,...,x_k$ are variable symbols, $X\in\mathrm{VAR}_{\mathrm{SO}}$ a $k$-ary relation variable
symbol and $\varphi\in S$,
then $I_{Xx_1,...,x_k}\varphi\ \in\ S$.
\item
If $x_1,...,x_k$ are variable symbols, $R\in\sigma$ a $k$-ary relation
symbol and $\varphi\in S$,
then $D_{Rx_1,...,x_k}\varphi\ \in\ S$.
\item
If $x_1,...,x_k$ are variable symbols, $X\in\mathrm{VAR}_{\mathrm{SO}}$ a $k$-ary
relation variable symbol and $\varphi\in S$,
then $D_{Xx_1,...,x_k}\varphi\ \in\ S$.
\item
If $\varphi\in S$ and $k\in\mathbb{N}$, then $k\varphi\in S$.
\end{enumerate}
While we could develop a sensible semantics for the language $\mathcal{L}^*(\sigma)$,
we shall only consider a sublanguage $\mathcal{L}(\sigma)\, \subseteq\, \mathcal{L}^*(\sigma)$
that avoids certain undesirable situations in semantic games.
Let $\varphi\in\mathcal{L}^*(\sigma)$ be a formula. Assume that 
$\varphi$ contains an atomic subformula $k\in\mathbb{N}$ and
another subformula $k\, \psi$. Assume that $k$ is \emph{not} a subformula of $k\, \psi$.
Then we say that $\varphi$ has a \emph{non-standard jump}.
Note that we define that \emph{every instance} of the syntactically
same subformula of $\varphi$ is a
\emph{distinct} subformula: for example, the formula
$(P(x)\wedge P(x))$ is considered to have \emph{three}
subformulae, these being the left and right instances of $P(x)$
and the formula $(P(x)\wedge P(x))$ itself. 
Thus for example the formula
$\bigl(\, k(P(x) \wedge k)\ \wedge\ k\, (P(x)\, \wedge\, k)\, \bigr)$
has a non-standard jump.
We define $\mathcal{L}(\sigma)$ to be the largest subset of $\mathcal{L}^*(\sigma)$
that does not contain formulae with non-standard jumps.
The reason we wish to avoid non-standard jumps is simple and
will become entirely clear when we define the semantics of $\mathcal{L}(\sigma)$
in Section \ref{sectionthree}\hspace{0.4mm}. Let us consider an example that demonstrates
the undesirable situation. Consider the formula $\bigl(k\, \wedge\, \exists x\, k\, P(x)\bigr)$
of $\mathcal{L}^*(\sigma)$. As will become clear in Section \ref{sectionthree}\hspace{0.4mm}, it is possible to end up
in the related semantic game in a position involving the atomic formula $P(x)$ \emph{without} first
visiting a position involving the formula $\exists x\, k\, P(x)$. This is
undesirable, since a related \emph{variable assignment function} will then not necessarily give
any value to the variable $x$.
For this reason we limit attention to the fragment $\mathcal{L}(\sigma)$
containing only formulae without non-standard jumps.
%
%
%

%
Before defining the semantics of the language $\mathcal{L}(\sigma)$, we
make a number of auxiliary definitions. Let $\mathfrak{A}$, $\mathfrak{B}$, etc.,  be models.
We let $A$, $B$, etc.,  denote the domains of the models in the usual way.
A function $f$ that interprets a
finite subset of $\mathrm{VAR}\cup\mathrm{VAR}_{\mathrm{SO}}$
in the domain of a model $\mathfrak{A}$ is called an \emph{assignment}.
Naturally,  if $X\in\mathrm{VAR}_{\mathrm{SO}}\cap\mathit{Dom}(f)$ is a $k$-ary
relation variable, then $f(X) \subseteq A^k$, and if $x\in\mathrm{VAR}\cap\mathit{Dom}(f)$, then $f(x)\in A$.
We let $f[x\mapsto a]$ denote the valuation with the domain $\mathit{Dom}(f)\cup\{x\}$
defined such that
$f[x\mapsto a](x) = a$
and
$f[x\mapsto a](y) = f(y)$ if $y\not= x$.
%
%
%
%
%
%
%
We analogously define $f[X\mapsto S]$, where $X\in\mathrm{VAR}_{\mathrm{SO}}$ is a $k$-ary
relation variable and $S\subseteq A^k$. We will also construct valuations of, say, the type
$f[x\mapsto a, y\mapsto b, X\mapsto S]$. The interpretation of these constructions is clear.
We define the set of free variables $\mathit{free}(\varphi)$ of a formula $\varphi\in\mathcal{L}(\sigma)$
as follows.
\begin{enumerate}
\item
If $R\in\sigma$, then $\mathit{free}(R(x_1,...,x_k)) = \{x_1,...,x_k\}$.
\item
If $X\in\mathrm{VAR}_{\mathrm{SO}}(k)$, then $\mathit{free}(X(x_1,...,x_k)) = \{X\}\cup\{x_1,...,x_k\}$.
\item
$\mathit{free}(x=y) = \{x,y\}$.
\item
$\mathit{free}(k) = \emptyset$.
\item
$\mathit{free}(\neg\varphi) = \mathit{free}(\varphi)$.
\item
$\mathit{free}((\varphi\wedge\psi)) = \mathit{free}(\varphi)\cup\mathit{free}(\psi)$.
\item
$\mathit{free}(\exists x\varphi) = \mathit{free}(\varphi)\setminus\{x\}$.
\item
$\mathit{free}(I x\, \varphi) = \mathit{free}(\varphi)\setminus\{x\}$.
%
%
\item
$\mathit{free}(I_{Rx_1,...,x_k}\varphi) = \mathit{free}(\varphi)\setminus\{x_1,...,x_k\}$.
\item
$\mathit{free}(I_{Xx_1,...,x_k}\varphi) = \mathit{free}(\varphi)\setminus\{X,x_1,...,x_k\}$.
\item
$\mathit{free}(D_{Rx_1,...,x_k}\varphi) = \mathit{free}(\varphi)\setminus\{x_1,...,x_k\}$.
\item
$\mathit{free}(D_{Xx_1,...,x_k}\varphi) = \mathit{free}(\varphi)\setminus\{X,x_1,...,x_k\}$.
\item
$\mathit{free}(k\varphi) = \mathit{free}(\varphi)$.
\end{enumerate}
A formula $\varphi$ of $\mathcal{L}(\sigma)$ is a
\emph{sentence} if $\mathit{free}(\varphi) = \emptyset$.
%

%
%

%
\section{A Semantics for $\mathcal{L}(\sigma)$}\label{sectionthree}
In this section we define a game-theoretic semantics for the language $\mathcal{L}(\sigma)$.
The semantics extends the well-known game-theoretic semantics of first-order logic
(see, e.g., \cite{mann}).
The semantic games are played by two players $\exists$ and $\forall$.
Let $\varphi$ be a formula of $\mathcal{L}(\sigma)$.
Let $\mathfrak{A}$ be a $\sigma$-model, and let
$f$ be an assignment that interprets the free variables of $\varphi$ in $A$.
Let $\#\in\{+,-\}$ be simply a symbol.
The quadruple $(\mathfrak{A},f,\#,\varphi)$ defines a semantic game $G(\mathfrak{A},f,\#,\varphi)$.
The set of \emph{positions} in the game $G(\mathfrak{A},f,\#,\varphi)$ is the smallest set $S$
such that the following conditions hold.
\begin{enumerate}
\item
$(\mathfrak{A},f,\#,\varphi)\in S$.
\item
If $(\mathfrak{B},g,\#',\neg\psi)\in S$,
then $(\mathfrak{B},g,\#'',\psi)\in S$,
where $\#''\  \in\ \{+,-\}\setminus\{\#'\}$.
\item
If $(\mathfrak{B},g,\#',(\psi\wedge\psi'))\in S$,
then $(\mathfrak{B},g,\#',\psi)\in S$ and $(\mathfrak{B},g,\#',\psi')\in S$.
\item
If $(\mathfrak{B},g,\#',\exists x\psi)\in S$ and $a\in B$,
then $(\mathfrak{B},g[x\mapsto a],\#',\psi)\in S$.
\item
If $(\mathfrak{B},g,\#',I x\, \psi)\in S$ and $b\not\in B$ is a
fresh element\hspace{0.4mm}\footnote{To avoid introducing a proper class of new positions
here, we assume $b = B$. Since $B\not\in B$, the 
element $b = B$ is a fresh element.
Only a single new position is generated.}\hspace{0.4mm},
then $(\mathfrak{B}\cup\{b\},g[x\mapsto b],\#',\psi)\in S$;
we define $\mathfrak{B}\cup\{b\}$ to be the $\sigma$-model $\mathfrak{C}$ where
$b$ is simply a fresh isolated point, i.e., the domain of $\mathfrak{C}$ is $B\cup\{b\}$,
and $R^{\mathfrak{C}} = R^{\mathfrak{B}}$ for each $R\in\sigma$.
\item
If $(\mathfrak{B},g,\#',I_{Rx_1,...,x_k}\psi)\in S$
and $b_1,...,b_k\in B$, then $(\mathfrak{B}^*,g^*,\#',\psi)\in S$,
where $\mathfrak{B}^*$ is obtained from $\mathfrak{B}$ by defining
$R^{\mathfrak{B}^*} := R^{\mathfrak{B}}\cup\{(b_1,...,b_k)\}$,
and $g^* := g[x_1\mapsto b_1,...,x_k\mapsto b_k]$.
For each relation symbol $P\in\sigma\setminus\{R\}$,
we have $P^{\mathfrak{B}^*} := P^{\mathfrak{B}}$.
The models $\mathfrak{B}$ and $\mathfrak{B}^*$
have the same domain.
\item
Assume $(\mathfrak{B},g,\#',I_{Xx_1,...,x_k}\psi)\in S$ and
$b_1,...,b_k\in B$. If $X\in\mathit{Dom}(g)$,
call $C := g(X)$. Otherwise let $C := \emptyset$.
Then $(\mathfrak{B},g^*,\#',\psi)\in S$,
where $g^* := g[x_1\mapsto b_1,...,x_k\mapsto b_k, X\mapsto (C\cup\{(b_1,...,b_k)\})]$.
\item 
If $(\mathfrak{B},g,\#',D_{Rx_1,...,x_k}\psi)\in S$
and $b_1,...,b_k\in B$, then $(\mathfrak{B}^*,g^*,\#',\psi)\in S$,
where $\mathfrak{B}^*$ is obtained from $\mathfrak{B}$ by defining
$R^{\mathfrak{B}^*} := R^{\mathfrak{B}}\setminus\{(b_1,...,b_k)\}$,
and $g^* := g[x_1\mapsto b_1,...,x_k\mapsto b_k]$.
For each relation symbol $P\in\sigma\setminus\{R\}$,
we have $P^{\mathfrak{B}^*} := P^{\mathfrak{B}}$.
The models $\mathfrak{B}$ and $\mathfrak{B}^*$
have the same domain.
\item
Assume $(\mathfrak{B},g,\#',D_{Xx_1,...,x_k}\psi)\in S$ and
$b_1,...,b_k\in B$. If $X\in\mathit{Dom}(g)$,
call $C := g(X)$. If
$X\not\in\mathit{Dom}(g)$,
define $C := \emptyset$.
Then $(\mathfrak{B},g^*,\#',\psi)\in S$,
where $g^* := g[x_1\mapsto b_1,...,x_k\mapsto b_k, X\mapsto (C\setminus\{(b_1,...,b_k)\})]$.
\item
If $(\mathfrak{B},g,\#',k\psi)\in S$,
then $(\mathfrak{B},g,\#',\psi)\in S$. 
\end{enumerate}
The game $G(\mathfrak{A},f,\#,\varphi)$ is played as follows.
\begin{enumerate}
\item
Every \emph{play} of the game begins from the position $(\mathfrak{A},f,\#,\varphi)$.
\item
If a position $(\mathfrak{B},g,\#',\neg\psi)$ is reached in a play of the game,
the play continues from the position $(\mathfrak{B},g,\#'',\psi)$,
where $\#''\in\{+,-\}\setminus\{\#'\}$.
\item
If a position $(\mathfrak{B},g,\#',(\psi\wedge\psi'))$ is reached,
then the play continues as follows.
If $\#' = +$ (respectively, $\#' = -$),
then the player $\forall$ (respectively, $\exists$) picks a formula $\chi\in\{\psi,\psi'\}$,
and the play continues from the position $(\mathfrak{B},g,\#',\chi)$.
\item
If a position $(\mathfrak{B},g,\#',\exists x\psi)$ is reached,
then the play continues as follows.
If $\#' = +$ (respectively, $\#' = -$),
then the player $\exists$ (respectively, $\forall$) picks an element $b\in B$,
and the play continues from the position $(\mathfrak{B},g[x\mapsto b],\#',\psi)$.
\item
If a position $(\mathfrak{B},g,\#',I x\, \psi)$ is reached,
then the play continues from the position $(\mathfrak{B}\cup\{b\},g[x\mapsto b],\#',\psi)$,
where $\mathfrak{B}\cup\{b\}$ is the $\sigma$-model $\mathfrak{C}$, where
$b$ is simply a fresh isolated point\footnote{Recall that we let $b := B$ in order to avoid
proper classes of new positions.}\hspace{0.4mm}, i.e., the domain of $\mathfrak{C}$ is $B\cup\{b\}$,
and $R^{\mathfrak{C}} = R^{\mathfrak{B}}$ for each $R\in\sigma$.
\item
Assume a position $(\mathfrak{B},g,\#',I_{Rx_1,...,x_k}\psi)$ has been
reached. The play of the game continues as follows.
If $\#' = +$ (respectively, $\#' = -$), then the player $\exists$ (respectively, $\forall$)
chooses a tuple $(b_1,...,b_k)\in B^k$.
The play of the game continues from the position $(\mathfrak{B}^*,g^*,\#',\psi)$,
where $\mathfrak{B}^*$ is obtained from $\mathfrak{B}$ by redefining
$R^{\mathfrak{B}^*} := R^{\mathfrak{B}}\cup\{(b_1,...,b_k)\}$,
and $g^* := g[x_1\mapsto b_1,...,x_k\mapsto b_k]$.
Other relations and the domain remain unaltered.
\item
Assume a position $(\mathfrak{B},g,\#',I_{Xx_1,...,x_k}\psi)$ has been
reached. The play of the game continues as follows.
If $\#' = +$ (respectively, $\#' = -$), then the player $\exists$ (respectively, $\forall$)
chooses a tuple $(b_1,...,b_k)\in B^k$.
The play of the game continues from the position $(\mathfrak{B},g^*,\#',\psi)$,
where $g^* := g[x_1\mapsto b_1,...,x_k\mapsto b_k, X\mapsto (C\cup\{(b_1,...,b_k)\})]$;
here $C = g(X)$ if $X\in\mathit{Dom}(g)$, and otherwise $C = \emptyset$.
\item
Assume a position $(\mathfrak{B},g,\#',D_{Rx_1,...,x_k}\psi)$ has been
reached. The play of the game continues as follows.
If $\#' = +$ (respectively, $\#' = -$), then the player $\exists$ (respectively, $\forall$)
chooses a tuple $(b_1,...,b_k)\in B^k$.
The play of the game continues from the position $(\mathfrak{B}^*,g^*,\#',\psi)$,
where $\mathfrak{B}^*$ is obtained from $\mathfrak{B}$ by redefining
$R^{\mathfrak{B}^*} := R^{\mathfrak{B}}\setminus\{(b_1,...,b_k)\}$,
and $g^* := g[x_1\mapsto b_1,...,x_k\mapsto b_k]$.
Other relations and the domain remain unaltered.
\item
Assume a position $(\mathfrak{B},g,\#',D_{Xx_1,...,x_k}\psi)$ has been
reached. The play of the game continues as follows.
If $\#' = +$ (respectively, $\#' = -$), then the player $\exists$ (respectively, $\forall$)
chooses a tuple $(b_1,...,b_k)\in B^k$.
If $X\in\mathit{Dom}(g)$, call $C:=g(X)$.
Otherwise define $C := \emptyset$.
The play of the game continues from the position $(\mathfrak{B},g^*,\#',\psi)$,
where $g^* := g[x_1\mapsto b_1,...,x_k\mapsto b_k, X\mapsto (C\setminus\{(b_1,...,b_k)\})]$.
\item
If a position $(\mathfrak{B},g,\#',k\psi)$ is reached,
then the play of the game continues
from the position $(\mathfrak{B},g,\#',\psi)$.
\item
If a position $(\mathfrak{B},g,\#',k)$ is reached,
then the play of the game continues as follows.
If $\#' = +$ (respectively, $\#' = -$)
and there exists a subformula
$k\psi$ of the original formula $\varphi$,
then the player $\exists$ (respectively, $\forall$) chooses some
subformula $k\chi$ of $\varphi$, and 
the play continues from the position
$(\mathfrak{B},g,\#',k\chi)$.
If no subformula $k\psi$ exists,
the play of the game ends.
\item
If $\psi$ is an atomic formula $R(x_1,...,x_k)$,
$X(x_1,...,x_k)$ or $x = y$, and a position $(\mathfrak{B},g,\#',\psi)$ is reached,
then the play of the game ends.
\end{enumerate}
A play of the game $G(\mathfrak{A},f,\#,\varphi)$ is played up to omega
rounds. If a play of the game continues for omega rounds, then neither of the two players wins the play.
If a play of the game ends after a finite number of rounds, then one
of the players wins the play. The winner is determined as follows.
\begin{enumerate}
\item
If the play ends in a position $(\mathfrak{B},g,\#',k)$, which may happen
in the pathological case where there are no subformulae of $\varphi$ of the type $k\psi$, then
$\exists$ wins if $\#' = -$ and $\forall$ wins if $\#' = +$.
\item
If the play ends in a position $(\mathfrak{B},g,\#',\psi)$, where
$\psi$ is an atomic formula $R(x_1,...,x_k)$, $X(x_1,...x_{k})$ or $x=y$,
then the winner of the play is determined as follows.
\begin{enumerate}
\item
Assume $\#' = +$. Then $\exists$ wins if
$\mathfrak{B},g\models\psi$. If $\mathfrak{B},g\not\models\psi$,
then $\forall$ wins. Here $\models$ is the semantic
turnstile of ordinary first-order logic.
\item
Assume $\#' = -$. Then $\forall$ wins if
$\mathfrak{B},g\models\psi$. If $\mathfrak{B},g\not\models\psi$,
then $\exists$ wins.
\end{enumerate}
\end{enumerate}
A \emph{strategy} of
$\exists$ in the game $G(\mathfrak{A},f,\#,\varphi)$ is
simply a function that determines a unique choice for the player $\exists$
in every position of the game $G(\mathfrak{A},f,\#,\varphi)$ that
requires $\exists$ to make a choice. A strategy of $\forall$ is
defined analogously. A strategy of $\exists$ ($\forall$) in
the game $G(\mathfrak{A},f,\#,\varphi)$ is a \emph{winning strategy}
if every play of the game where $\exists$ ($\forall$) makes
her moves according to the strategy, ends after a finite number 
of rounds in a position where $\exists$ ($\forall$) wins.
We write $\mathfrak{A},f\models^+\varphi$ iff the player $\exists$
has a winning strategy in the game $G(\mathfrak{A},f,+,\varphi)$.
We write $\mathfrak{A},f\models^-\varphi$ iff $\exists$
has a winning strategy in the game $G(\mathfrak{A},f,-,\varphi)$.
By duality of the rules of the game, it is easy to see that $\exists$
has a winning strategy in $G(\mathfrak{A},f,-,\varphi)$ iff $\forall$
has a winning strategy in $G(\mathfrak{A},f,+,\varphi)$. Similarly, $\exists$
has a winning strategy in $G(\mathfrak{A},f,+,\varphi)$ iff $\forall$
has a winning strategy in $G(\mathfrak{A},f,-,\varphi)$.
Let $\varphi$ be a \emph{sentence} of $\mathcal{L}(\sigma)$.
We write $\mathfrak{A}\models^+\varphi$ iff $\mathfrak{A},\emptyset\models^+\varphi$,
where $\emptyset$ denotes the empty valuation. Similarly, we write $\mathfrak{A}\models^-\varphi$
iff $\mathfrak{A},\emptyset\models^-\varphi$.
\section{Turing-Completeness}\label{four}
Let $\sigma$ be a finite nonempty set of unary relation symbols. let $\mathit{Succ}$ be a
binary relation symbol.
A \emph{word model} $\mathfrak{A}$ over 
the vocabulary $\{\mathit{Succ}\}\cup\sigma$ is defined as follows.
\begin{enumerate}
\item
The domain of $\mathfrak{A}$ is a
nonempty finite set.
\item
The binary predicate $\mathit{Succ}$ is a successor relation over $A$, i.e., a
binary relation corresponding to a linear order, but with maximum out-degree
and in-degree equal to one.
\item
Let $b\in A$ denote the smallest element with respect to $\mathit{Succ}$.
We have $b\not\in P^{\mathfrak{A}}$ for each $P\in\sigma$. (This is because 
we do not want to consider models with the empty domain; the empty word
will correspond to the word model with exactly one element.)
For each element $a\in A\setminus\{b\}$, there exists exactly one predicate $P\in\sigma$
such that $a\in P^{\mathfrak{A}}$.
\end{enumerate}
Word models canonically encode finite words. For example 
the word $abbaa$ over the alphabet $\{a,b\}$ is encoded by the word model $\mathfrak{M}$
over the vocabulary $\{\mathit{Succ}, P_a,P_b\}$ defined as follows.
\begin{enumerate}
\item
$M = \{0,...,5\}$.
\item
$\mathit{Succ}^{\mathfrak{M}}$ is the canonical successor relation
on $M$.
\item
$P_a^{\mathfrak{M}} = \{1,4,5\}$
and 
$P_b^{\mathfrak{M}} = \{2,3\}$.
\end{enumerate}
If $w$ is a finite word, we let $\mathcal{M}(w)$ denote its encoding
by a word model in the way defined above. If $W$ is a set of finite words,
then $\mathcal{M}(W) = \{\, \mathcal{M}(w)\, |\, w\in W\, \}$.
If $\Sigma$ is a finite nonempty alphabet, we let $\mathcal{M}(\Sigma)$
denote the vocabulary $\{\, \mathit{Succ}\, \}\cup\{\, P_a\, |\, a\in\Sigma\, \}$.
We define computation of Turing machines in the standard way that involves a possible
\emph{tape alphabet} in addition to an \emph{input alphabet}.
These two alphabets are disjoint.
Let $\Sigma$ be a finite nonempty alphabet.
Then $\Sigma^*$ is the set of all inputs to a
Turing machine $\mathrm{TM}$ whose input alphabet is $\Sigma$.
During computation, $\mathrm{TM}$ may employ an additional finite set $S$
of tape symbols. That set $S$ is the tape alphabet of $\mathrm{TM}$.
There is a nice loose analogy
between tape alphabet symbols of Turing machines and relation variable
symbols in $\mathrm{VAR}_{\mathrm{SO}}$
used in formulas of $\mathcal{L}$.
\begin{theorem}\label{firsttheorem}
Let $\Sigma$ be a finite nonempty alphabet.
Let $\mathrm{TM}$ be a deterministic Turing machine with the input
alphabet $\Sigma$. Then there exists a sentence $\varphi_{\mathrm{TM}}\in\mathcal{L}(\mathcal{M}(\Sigma))$
such that the following conditions hold.
\begin{enumerate}
\item
Let $W\subseteq\Sigma^*$ be the set of words $w$ such
that $\mathrm{TM}$ halts in an accepting state with the input $w$.
Then for all $w\in\Sigma^*$, $\mathcal{M}(w)\models^+\varphi_{\mathrm{TM}}$ iff $w\in W$.
\item
Let $U\subseteq\Sigma^*$ be the set of words $u$ such
that $\mathrm{TM}$ halts in a rejecting state with the input $u$.
Then for all $w\in\Sigma^*$, $\mathcal{M}(w)\models^-\varphi_{\mathrm{TM}}$ iff $w\in U$.
\end{enumerate}
\end{theorem}
\begin{proof}
%
%
We shall define a sentence $\varphi_{\mathrm{TM}}$ such that the
semantic games involving $\varphi_{\mathrm{TM}}$ simulate
the action of $\mathrm{TM}$.
Let $Q$ be the set of states of $\mathrm{TM}$.
For each $q\in Q$, reserve a variable symbol $x_q$.
Furthermore, let $y_{\mathit{state}}$ be a variable symbol.
Intuitively, the equality $y_{\mathit{state}} = x_q$ will
hold in the semantic game $G(\mathcal{M}(w),\emptyset,+,\varphi_{\mathrm{TM}})$
exactly when $\mathrm{TM}$ is in the state $q$ during a
run with the input $w$.
Simulating the action of the head of the Turing machine $\mathrm{TM}$ is a bit
more complicated, since when defining the 
new position of the head with a subformula of $\varphi_{\mathrm{TM}}$, information
concerning the old
position must be somehow accessible.\hspace{0.4mm}\footnote{Note that we
assume, w.l.o.g., that $\mathrm{TM}$ has a single head.}
Fix \emph{two} variables $x_{\mathit{head}}^1$ and $x_{\mathit{head}}^2$.
These variables will encode the position of the head. Define three further variables
$y_{\mathit{head}}^1$, $y_{\mathit{head}}^2$, and $y_{\mathit{head}}$.
The tape of $\mathrm{TM}$ will be encoded by the (dynamically extendible) successor relation $\mathit{Succ}$,
which is a part of the model (or models, to be exact) constructed during the semantic game.
The variables $x_{\mathit{head}}^1$ and $x_{\mathit{head}}^2$ will denote
elements of the successor relation. Intuitively, $y_{\mathit{head}} = y_{\mathit{head}}^1$
will mean that $x_{\mathit{head}}^1$ indicates the current position of the head of $\mathrm{TM}$,
while $y_{\mathit{head}} = y_{\mathit{head}}^2$ will mean that $x_{\mathit{head}}^2$ 
indicates the position of the head of $\mathrm{TM}$.
The value of $x_{\mathit{head}}^1$ will always be
easily definable based on the value of $x_{\mathit{head}}^2$, 
and vice versa, the value of $x_{\mathit{head}}^2$ will be
definable based on the value of $x_{\mathit{head}}^1$.
If $\mathrm{TM}$ employs tape alphabet symbols $s\not\in\Sigma$,
these can be encoded by unary relation variables $X_s$. Intuitively, 
if $u$ is an element of the domain of the model under investigation, then $X_s(u)$
will mean that the point of the tape of $\mathrm{TM}$ corresponding to $u$
contains the symbol $s$. Similarly, for an input alphabet symbol $t\in\Sigma$\hspace{0.3mm},
\ $P_t(u)$
will mean that the point of the tape of $\mathrm{TM}$ corresponding to $u$
contains the symbol $t$.
The sentence $\varphi_{\mathrm{TM}}$ will contain subformulae which are \emph{essentially}
(but not exactly, as we shall see) of the type
$$\bigl(\psi_{\mathit{state}}\wedge\psi_{\mathit{tape\_position}}\bigr)\ \rightarrow\
\bigl(\psi_{\mathit{new\_state}}\wedge\psi_{\mathit{new\_tape\_position}}\wedge \mathit{loop}\bigr),$$
where $\mathit{loop}$ is simply the atomic formula $1$, which indicates
that the semantic game ought to be continued from some subformula $1\psi$ of $\varphi_{\mathrm{TM}}$.
The sentence $\varphi_{\mathrm{TM}}$ will also contain subformulae which are \emph{essentially} of the type
$$\bigl(\psi_{\mathit{state}}\wedge\psi_{\mathit{tape\_position}}\bigr)\ \rightarrow\
\bigl(\psi_{\mathit{new\_final\_state}}\wedge\psi_{\mathit{new\_tape\_position}}\wedge \top\bigr)$$
and
$$\bigl(\psi_{\mathit{state}}\wedge\psi_{\mathit{tape\_position}}\bigr)\ \rightarrow\
\bigl(\psi_{\mathit{new\_final\_state}}\wedge\psi_{\mathit{new\_tape\_position}}\wedge \bot\bigr)$$
where in the first case the final state is an
accepting state, and in the second case a
rejecting state. Here $\top$ ($\bot$) is the formula $\forall x\, x=x$\,  ($\neg\forall x\, x=x $).
Let $s,t\in\Sigma$ be input alphabet symbols of $\mathrm{TM}$.
Consider a transition instruction of $\mathrm{TM}$ of the type
$T(q_i,s) = (q_j,t,\mathit{right})$, which states that if the state is $q_i$
and the symbol scanned is $s$, then write $t$ to the current cell,
change state to $q_j$, and move right.
Let us call this instruction $\mathit{instr}$.
The instruction $\mathit{instr}$ defines a formula $\psi_{\mathit{instr}}$.
Assume $q_j$ is \emph{not} a final state.
Let us see how $\psi_{\mathit{instr}}$ is constructed.
Define the formula $\psi_{\mathit{state}}^{q_i} := y_{\mathit{state}} = x_{q_i}$.
Define the formula $\psi_{\mathit{symbol}}^s$ to be the conjunction of the following formulae.
\begin{enumerate}
\item
$y_{\mathit{head}} = y_{\mathit{head}}^1\ \rightarrow\ P_s(x_{\mathit{head}}^1)$,
\item
$y_{\mathit{head}} = y_{\mathit{head}}^2\ \rightarrow\ P_s(x_{\mathit{head}}^2)$.
\end{enumerate}
Define $\chi_1'$ to be the formula
\begin{multline*}
D_{P_s\, x}\, I_{P_t\, y}\, \exists x_{\mathit{head}}^2\exists y_{\mathit{head}}\exists y_{\mathit{state}}
\bigl(\ x = x_{\mathit{head}}^1\, \wedge\, y = x_{\mathit{head}}^1\, \wedge\, 
y_{\mathit{head}} = y_{\mathit{head}}^2\, \wedge\, y_{\mathit{state}} = x_{q_j}\, \wedge\, \chi'\, \wedge\, 1\ \bigr),
\end{multline*}
where $\chi'$ is a formula that forces $x_{\mathit{head}}^2$ to be interpreted as 
the successor of $x_{\mathit{head}}^1$ with respect to $\mathit{Succ}$.
It is possible that no successor of $x_{\mathit{head}}^1$ exists in the current model.
In that case a successor can be constructed by appropriately using the
operators $Iz$ and $I_{\mathit{Succ\, uv}}$.
To cover this case, define $\chi_1''$ to be the formula
\begin{multline*}
D_{P_s\, x}\, I_{P_t\, y}\, Iz\  I_{\mathit{Succ\, uv}}\,
\exists x_{\mathit{head}}^2\exists y_{\mathit{head}}\exists y_{\mathit{state}}
\bigl(\ x = x_{\mathit{head}}^1\, \wedge\, y = x_{\mathit{head}}^1\, \wedge\, 
y_{\mathit{head}} = y_{\mathit{head}}^2\, \wedge\, y_{\mathit{state}} = x_{q_j}\, \wedge\, \chi'\,
\wedge\, \chi''\, \wedge\, 1\ \bigr),
\end{multline*}
where $\chi''$ forces the fresh point $z$ to be the
successor of $x_{\mathit{head}}^1$ with respect to $\mathit{Succ}$,
and $\chi'$ forces  $x_{\mathit{head}}^2$ to be
the successor of $x_{\mathit{head}}^1$.
Let $\alpha$ be a formula that states that $x_{\mathit{head}}^1$
has a successor with respect to $\mathit{Succ}$ in the current model.
Define $\chi_1$ to be the conjunction
$(\alpha\rightarrow\chi_1')\wedge(\neg\alpha\rightarrow\chi_1'')$.
The formula $\chi_1$ simulates the instruction $\mathit{instr}$ when
the current position of the head of $\mathrm{TM}$ is encoded by $x_{\mathit{head}}^1$.
The formula determines a new position for $x_{\mathit{head}}^2$ based
on the current position of $x_{\mathit{head}}^1$.
A similar formula $\chi_2$ can be defined analogously to deal with
the situation where the current position of the head is encoded by $x_{\mathit{head}}^2$.
Define $\beta$ to be the conjunction of the formulae
\begin{enumerate}
\item
$y_{\mathit{head}} = y_{\mathit{head}}^1\ \rightarrow\ \chi_1$,
\item
$y_{\mathit{head}} = y_{\mathit{head}}^2\ \rightarrow\ \chi_2$.
\end{enumerate}
Define $\psi_{\mathit{instr}}$ to be the formula
$\bigl(\psi_{\mathit{state}}^{q_i}\wedge\psi_{\mathit{symbol}}^{s}\bigr)\ \ \ \rightarrow\ \ \ \beta.$
Formulae $\psi_{\mathit{instr}'}$, where $\mathit{instr}'$ tells $\mathrm{TM}$
to move to a final state, are defined similarly, but do not have the atom $1$.
Instead, accepting states have the atom $\top$ and rejecting states the atom $\bot$.
We shall not explicitly discuss for example instructions where the head is to move left,
since all possible instructions can be easily specified by
formulae analogous to the ones above.
Recall that $Q$ is the set of states of $\mathrm{TM}$.
Let $q_1,...,q_n$ enumerate the elements of $Q$. Define
$$I{\overline{x}}\ :=\ Iy_{\mathit{head}}^1\, 
Iy_{\mathit{head}}^2\, 
Ix_{q_1}....Ix_{q_n}.$$
Let $\mathbb{I}$ be the set of instructions of $\mathrm{TM}$.
The sentence $\varphi_{\mathrm{TM}}$ is the formula
$$I\overline{x}\ \exists y_{\mathit{head}}\exists x_{\mathit{head}}^1
\exists x_{\mathit{head}}^2\exists y_{\mathit{state}}
\bigl(\ \psi_{\mathit{initial}}\ \wedge\ 1\bigl(\bigwedge
\limits_{\mathit{instr}\, \in\, \mathbb{I}}\psi_{\mathit{instr}}\bigr)\ \bigr),$$
where $\psi_{\mathit{initial}}$ states that the following
conditions hold.
\begin{enumerate}
\item
$y_{\mathit{state}}$
is equal to $x_{q}$, where $q$ is the beginning state of $\mathrm{TM}$.
\item
$y_{\mathit{head}}$ is equal to $y_{\mathit{head}}^1$.
\item
$x_{\mathit{head}}^1$ is interpreted as the point corresponding
to the beginning position of the head of $\mathrm{TM}$.
\end{enumerate}
It is not difficult to see that $\varphi_{\mathrm{TM}}$ corresponds to $\mathrm{TM}$
in the desired way.
\end{proof}
We then prove that every sentence of $\mathcal{L}$ spefifying a
property of word models can be simulated by a Turing machine.
For this purpose, we use K\"{o}nig's Lemma.
\begin{lemma}[K\"{o}nig]
Let $T$ be a finitely branching tree with infinitely many nodes.
Then $T$ contains an infinite branch.
\end{lemma}
In the following, \emph{accepting} means halting in an
accepting state, and \emph{rejecting} means halting in a
rejecting (i.e., non-accepting) state.
\begin{theorem}\label{secondtheorem}
Let $\Sigma$ be a finite nonempty alphabet.
Let $\varphi$ be a sentence of $\mathcal{L}(\mathcal{M}(\Sigma))$.
Then there exists a deterministic Turing machine $\mathrm{TM}$ such
that the following conditions hold.
\begin{enumerate}
\item
Let $W\subseteq\Sigma^*$ be the set of words $w$ such
that $\mathcal{M}(w)\models^+\varphi$.
Then for all $w\in\Sigma^*$, $\mathrm{TM}$ accepts $w$ iff $w\in W$.
\item
Let $U\subseteq\Sigma^*$ be the set of words $w$ such
that $\mathcal{M}(w)\models^-\varphi$.
Then for all $w\in\Sigma^*$, $\mathrm{TM}$ rejects $w$ iff $w\in U$.
\end{enumerate}
\end{theorem}
\begin{proof}
%
%
Fix some positive integer $k$.
Given an input word $w$, the Turing machine $\mathrm{TM}$ first enumerates all plays of
$G(\mathcal{M}(w),\emptyset,+,\varphi)$ with $k$ rounds or less.
If $\exists$ wins such a play, $\mathrm{TM}$ checks whether
there is a winning strategy for 
$\exists$ that always leads to a win in $k$ or fewer rounds,
meaning that no play where $\exists$ follows the
strategy lasts for $k+1$ rounds or more, and $\exists$ wins all plays where she
follows her strategy.
Similarly, if $\forall$ wins a play with $k$ or fewer rounds, $\mathrm{TM}$ checks whether
there is a winning strategy for 
$\forall$ that always leads to a win in at most $k$ rounds.
If there is such a strategy for $\exists$ ($\forall$), then $\mathrm{TM}$ halts
in an accepting (rejecting) state.
If no winning strategy is found, the machine $\mathrm{TM}$ 
checks all plays with $k+1$ rounds. Again,
if $\exists$ wins such a play, $\mathrm{TM}$ checks whether
there is a winning strategy for 
$\exists$ that always leads to a win in at
most $k+1$ rounds, and similarly for $\forall$.
Again, if a winning strategy for $\exists$ ($\forall$) is found,
then $\mathrm{TM}$ halts in an accepting (rejecting) state.
If no winning strategy is found, the machine scans all plays of the
length $k+2$, and so on. This process of scanning increasingly
long plays is carried on potentially infinitely long.
Now assume, for the sake of contradiction, that $\exists$ ($\forall$) has a winning strategy
with arbitrarily long plays resulting from following the strategy.
Then the game tree restricted to paths where $\exists$ ($\forall$) follows
the strategy has infinitely many nodes.
Let $T$ denote the restriction of the game tree to paths where the strategy is followed.
Since each game position can have only finitely many successor positions,
and since $T$ is infinite, we conclude by K\"{o}nig's lemma
that $T$ has an infinite branch.
Thus the strategy of $\exists$ ($\forall$)
cannot be a winning strategy. This is a contradiction. Hence each winning strategy has a
finite bound $n$ such that each play where the strategy is followed,
goes on for at most $n$ rounds.
Thus $\mathrm{TM}$ has the desired properties.
The crucial issue here is that there exist a \emph{finite}
number of possible moves at every position of the game.
This finiteness is due to the underlying models always being finite and
properties of the operators of the logic $\mathcal{L}$.
\end{proof}
Note that our translations of Turing machines to formulae of
$\mathcal{L}$ and formulae of $\mathcal{L}$ to
Turing machines are both effective.
\section{Arbitrary Structures}\label{five}
Above we limited attention to word models. This is not necessary,
as Theorems \ref{firsttheorem} and \ref{secondtheorem}
can easily be generalized to the context of arbitrary finite structures.
In this section we show how this generalization can be done.
When investigating computations on structure classes (rather than strings),
Turing machines of course operate on 
\emph{encodings} of structures.
We will use the encoding scheme of \cite{libkin}.
Let $\tau$ be a finite relational vocabulary and $\mathfrak{A}$ a
finite $\tau$-structure. In order to encode the structure $\mathfrak{A}$ by a binary string,
we first need to define a linear ordering of the domain $A$ of $\mathfrak{A}$.
Let $<^{\mathfrak{A}}$ denote such an ordering.
Let $R\in\tau$ be a $k$-ary relation symbol.
The encoding $\mathit{enc}(R^{\mathfrak{A}})$
of $R^{\mathfrak{A}}$ is the $|A|^k$-bit string defined as follows.
Consider an enumeration of all $k$-tuples over $A$ in
the \emph{lexicographic order} defined with respect to $<^{\mathfrak{A}}$.
In the lexicographic order, $(a_1,...,a_k)$ is smaller than $(a_1',...,a_k')$
iff there exists $i\in\{1,...,k\}$ such that $a_i < a_i'$ 
and $a_j = a_j'$ for all $j < i$.
There are $| A |^k$ tuples in $A^k$. The string
$\mathit{enc}(R^{\mathfrak{A}})$
is the string
$t \in\{0,1\}^*$
of the length $| A |^k$
such that the bit $t_i$ of $t = t_1\, ...\, t_{| A |^k}$ is $1$
if and only if the $i$-th tuple $(a_1,...,a_k)\in A^{k}$ in the lexicographic order
is in the relation $R^{\mathfrak{A}}$.
The encoding $\mathit{enc}(\mathfrak{A})$ is defined as follows.
We first order the relations in $\tau$. Let $p$ be the number of
relations in $\tau$, and let $R_1,...,R_p$ enumerate the
symbols in $\tau$ according to the order.
We define
$$\mathit{enc}(\mathfrak{A})\ :=\ 0^{|A|}\cdot 1\cdot \mathit{enc}(R_1^{\mathfrak{A}})\cdot...\cdot
\mathit{enc}(R_p^{\mathfrak{A}}).$$
Notice that the encoding of $\mathfrak{A}$ depends on
the order $<^{\mathfrak{A}}$ and the ordering of the relation symbols in $\tau$.
Let $\mathcal{C}$ be the class of exactly all finite $\tau$-models.
Let $\mathcal{C}_+$, $\mathcal{C}_-$ and $\mathcal{C}_0$ 
be subclasses of $\mathcal{C}$ such that the following conditions hold.
\begin{enumerate}
\item
Each of the three classes
$\mathcal{C}_+$, $\mathcal{C}_-$ and $\mathcal{C}_0$ is closed under isomorphism.
\item
The classes are disjoint, i.e., the intersection
of any two of the three classes is empty.
\item
$\mathcal{C}_+\cup\mathcal{C}_-\cup\mathcal{C}_0\, =\, \mathcal{C}$.
\end{enumerate}
We say that $(\mathcal{C}_+,\mathcal{C}_-,\mathcal{C}_0)$
is a \emph{Turing classification} of finite $\tau$-models if there
exists a Turing machine $\mathrm{TM}$ such that the following conditions hold.
\begin{enumerate}
\item
The input alphabet of $\mathrm{TM}$ is $\{0,1\}$.
\item
$\mathrm{TM}$ rejects every input string that is 
not of the type $\mathit{enc}(\mathfrak{A})$ for any finite $\tau$-strucure $\mathfrak{A}$.
\item
There exists an ordering $<^{\tau}$ of $\tau$ such that the following conditions hold.
\begin{enumerate}
\item
Let $\mathfrak{A}\in\mathcal{C}$.
Let $\mathit{enc}(\mathfrak{A})$ and $\mathit{enc}'(\mathfrak{A})$ be
two encodings of $\mathfrak{A}$, both using
the order $<^{\tau}$ of $\tau$ but possibly a different
ordering of $A$.
Then one of the following three conditions holds.
\begin{enumerate}
\item
$\mathrm{TM}$ accepts both strings $\mathit{enc}(\mathfrak{A})$ and $\mathit{enc}'(\mathfrak{A})$.
\item
$\mathrm{TM}$ rejects both strings $\mathit{enc}(\mathfrak{A})$ and $\mathit{enc}'(\mathfrak{A})$.
\item
$\mathrm{TM}$ diverges on both input strings $\mathit{enc}(\mathfrak{A})$ and $\mathit{enc}'(\mathfrak{A})$.
\end{enumerate}
\item
Let $\mathfrak{A}\in\mathcal{C}$.
Let $\mathit{enc}(\mathfrak{A})$ be an encoding of $\mathfrak{A}$
according to the order $<^{\tau}$.
The following conditions hold.
\begin{enumerate}
\item
$\mathrm{TM}$ accepts $\mathit{enc}(\mathfrak{A})$ iff $\mathfrak{A}\in\mathcal{C}_+$.
\item
$\mathrm{TM}$ rejects $\mathit{enc}(\mathfrak{A})$ iff $\mathfrak{A}\in\mathcal{C}_-$.
\item
$\mathrm{TM}$ diverges on the input $\mathit{enc}(\mathfrak{A})$ iff $\mathfrak{A}\in\mathcal{C}_0$.
\end{enumerate}
\end{enumerate}
\end{enumerate}
We say that $\mathrm{TM}$ \emph{witnesses}
the Turing classification $(\mathcal{C}_+,\mathcal{C}_-,\mathcal{C}_0)$.
The logic $\mathcal{L}$ combines the expressivity of first-order logic with the possibility
of building fresh relations over fresh domain elements. The recursive looping capacity
enables a flexible way of using such fresh constructions. Therefore it is not
difficult to see that the following theorem holds.
\begin{theorem}\label{thirdtheorem}
Let $\tau$ be a finite relational vocabulary and 
$(\mathcal{C}_+,\mathcal{C}_-,\mathcal{C}_0)$ a
Turing classification of finite $\tau$-models.
Let $\mathrm{TM}$ be a Turing machine that
witnesses the classification $(\mathcal{C}_+,\mathcal{C}_-,\mathcal{C}_0)$.
Then there exists a sentence $\varphi_{\mathrm{TM}}$ of $\mathcal{L}(\tau)$ such that
the following conditions hold for finite $\tau$-models $\mathfrak{A}$.
\begin{enumerate}
\item
$\mathfrak{A}\models^+\varphi_{\mathrm{TM}}$ iff\, $\mathfrak{A}\in\mathcal{C}_+$\, .
\item
$\mathfrak{A}\models^-\varphi_{\mathrm{TM}}$ iff\, $\mathfrak{A}\in\mathcal{C}_-$\, .
\end{enumerate}
\end{theorem}
\begin{proof}[Proof sketch]
The simulation of a machine $\mathrm{TM}$ operating
on encodings of structures $\mathfrak{A}$ is done
by a sentence $\varphi_{\mathrm{TM}}$
of $\mathcal{L}$ as follows.
The ``input" to the formula $\varphi_{\mathrm{TM}}$ is a
finite $\tau$-structure $\mathfrak{A}$.
The formula $\varphi_{\mathrm{TM}}$ first uses $\mathfrak{A}$ in order to
construct a \emph{word model} $\mathfrak{M}_{\mathfrak{A}}$
that corresponds to a string $\mathit{enc}(\mathfrak{A})$ that encodes $\mathfrak{A}$.
The domains of $\mathfrak{M}_{\mathfrak{A}}$ and
$\mathfrak{A}$ are disjoint. The relation symbols of $\mathfrak{M}_{\mathfrak{A}}$
are symbols in $\mathrm{VAR}_{\mathrm{SO}}$, not symbols in $\tau$. 
Once $\mathfrak{M}_{\mathfrak{A}}$ has been constructed,
the formula $\varphi_{\mathrm{TM}}$ uses $\mathfrak{M}_{\mathfrak{A}}$
in order to simulate the computation of $\mathrm{TM}$ on
the string $\mathit{enc}(\mathfrak{A})$.
The simulation is done in the way described in the proof of
Theorem \ref{firsttheorem}\hspace{0.4mm}. 
The construction of the word model $\mathfrak{M}_{\mathfrak{A}}$
from $\mathfrak{A}$ is not difficult. First a fresh successor relation $S^{\mathfrak{A}}$
\emph{over the domain of\, $\mathfrak{A}$}
is constructed using the operator $I_{\mathit{S}\, xy}$.
The symbol $\mathit{S}$ is not in $\tau$. Instead, we use a
fresh symbol in $\mathrm{VAR}_{\mathrm{SO}}$.
Also, the successor symbol $\mathit{S}$ will not be part of the vocabulary of
the word model $\mathfrak{M}_{\mathfrak{A}}$. 
Let $<^{\mathfrak{A}}$ denote the linear order 
canonically associated with the successor relation $S^{\mathfrak{A}}$.
The order $<^{\mathfrak{A}}$, together with 
an ordering of $\tau$, define a string $\mathit{enc}(\mathfrak{A})$.
The model $\mathfrak{M}_{\mathfrak{A}}$ is the
word model corresponding to the string $\mathit{enc}(\mathfrak{A})$.
Due to the \emph{very high expressivity} of the logic $\mathcal{L}$,
is not difficult to build $\mathfrak{M}_{\mathfrak{A}}$
using $S^{\mathfrak{A}}$ and possibly further auxiliary relations.
Thus writing the formula $\varphi_{\mathrm{TM}}$ is
relatively straightforward.
We skip further details.
\end{proof}
\begin{theorem}\label{fourththeorem}
Let $\tau$ be a finite relational vocabulary.
Let $\varphi$ be a $\tau$-sentence of $\mathcal{L}$.
Then there exists a Turing classification $(\mathcal{C}_+,\mathcal{C}_-,\mathcal{C}_0)$
of finite $\tau$-models
such that for all finite $\tau$-models $\mathfrak{A}$, the following conditions hold.
\begin{enumerate}
\item
$\mathfrak{A}\in\mathcal{C}_+$ iff\, $\mathfrak{A}\models^+\varphi$.
\item
$\mathfrak{A}\in\mathcal{C}_-$ iff\, $\mathfrak{A}\models^-\varphi$.
\end{enumerate}
\end{theorem}
\begin{proof}
The proof is practically identical to the proof of Theorem \ref{secondtheorem}.
\end{proof}
\section{Generalized Quantifiers and Oracles}\label{six}
The relationship between oracles and Turing machines 
is analogous to the relationship between generalized quantifiers and logic.
Oracles allow arbitrary jumps in computations
in a similary way in which generalized quantifiers allow
the assertion of arbitrary properties of relations.
In this section we briefly discuss extensions of the logic $\mathcal{L}$
with generalized quantifiers. For the sake of simplicity, we only consider
\emph{unary quantifiers of the width one}, i.e., quantifiers of the type $(1)$.

\newcommand{\mA}{\mathfrak{A}}
\newcommand{\mB}{\mathfrak{B}}
\newcommand{\mM}{\mathfrak{M}}
\newcommand{\mfo}{\models_{\mathrm{FO}}}

A \emph{unary generalized quantifier of the width one}
(cf. \cite{lindstrom}) is a class $\mathcal{C}$ of structures
$(A,B)$ such that the following conditions hold.
\begin{enumerate}
\item
$A\not=\emptyset$ and $B\subseteq A$.
\item
If $(A',B')\in\mathcal{C}$ and if there is an isomorphism $f:A'\rightarrow A''$
from $(A',B')$ to another structure $(A'',B'')$,
then we have $(A'',B'')\in\mathcal{C}$.
\end{enumerate} 
Below the word \emph{quantifier} always means a
unary generalized quantifier of the width one.
Let $Q$ be a quantifier.
%
%
%
%
%
%
%
Let $\mathfrak{A}$ be a model 
with the domain $A$.
We define $Q^{\mathfrak{A}}\, :=\, 
\{\ B\ |\ (A,B)\in Q\ \}.$
%
%
%
%
%
%
%
%
%
Extend the formula formation rules of first-order logic
such that if $\varphi$ is a formula and $x$ a variable,
then $\hat{Q}x\, \varphi$ is a formula.
The operator $\hat{Q}x$ binds the variable $x$,
so the set of free variables of $\hat{Q}x\, \varphi$
is obtained by removing $x$ from the set of free variables of $\varphi$.
The standard semantic clause for the formula $\hat{Q}x\, \varphi$ is as follows.
Let $\mathfrak{A}$ be a model that interprets the
non-logical symbols in $\varphi$. Let $f$ be an assignment
function that interprets the free variables in $\hat{Q}x\, \varphi$.  
Then $\mathfrak{A},f\models \hat{Q}x\, \varphi$ iff
$\{\, a\in A\ |\ \mathfrak{A},f[x\mapsto a]\models \varphi\ \}\
\in\ Q^{\mathfrak{A}}.$
We then discuss how generalized quantifiers can be incorporated into the logic $\mathcal{L}$.
This simply amounts to extending the game-theoretic semantics such that generalized
quantifiers are taken into account.  This is accomplished in the canonical way
described below.\hspace{0.4mm}\footnote{Somewhat surprisingly, 
the semantic game moves for generalized quantifiers we are about to
define do not seem to have been 
defined in the exact same way in the literature before.
However, the article \cite{kuusisto} provides a rather similar but not exactly
the same treatment.}
Assume we have reached a position $(\mathfrak{A},f,+,\varphi)$
in a semantic game. If $Q^{\mathfrak{A}} = \emptyset$,
the player $\exists$ loses the play of the game.
Otherwise the player $\exists$ chooses a
set $S\in Q^{\mathfrak{A}}$. The player $\forall$ then chooses either
a point $s\in S$ of a point $s'\in A\setminus S$. (Here $A$ is
of course the domain of $\mathfrak{A}$.)
Suppose first that $\forall$ chooses $s\in S$. Then the game continues from the
position $(\mathfrak{A},f[x\mapsto s],+,\varphi)$.
Suppose then that $\forall$ chooses $s'\in A\setminus S$. Then the game continues from the
position $(\mathfrak{A},f[x\mapsto s'],-,\varphi)$.
The intuition behind these moves is that $\exists$ first chooses the set $S$ of
\emph{exactly all} witnesses for $\varphi$, and this set  $S$ must be in $Q^{\mathfrak{A}}$.
Then $\forall$ either opposes the claim that $S$ contains \emph{only} witnesses of $S$ by choosing a
potential counterexample $s\in S$, or alternatively, $\forall$ opposes the claim that $S$
contains \emph{all} witnesses of $\varphi$ by choosing a potential further witness $s'\in A\setminus S$.
Assume then that we have reached a position $(\mathfrak{A},f,-,\varphi)$
in a semantic game. If $Q^{\mathfrak{A}} = \emptyset$,
the player $\forall$ loses the play of the game.
Otherwise the player $\forall$ chooses a
set $S\in Q^{\mathfrak{A}}$. The player $\exists$ then chooses either
a point $s\in S$ of a point $s'\in A\setminus S$. 
Suppose that $\exists$ chooses $s\in S$. Then the game continues from the
position $(\mathfrak{A},f[x\mapsto s],-,\varphi)$.
Suppose then that $\exists$ chooses $s'\in A\setminus S$.
Then the game continues from the
position $(\mathfrak{A},f[x\mapsto s'],+,\varphi)$.
It is straightforward to prove that these rules give a semantics
such that in restriction to formulae of first-order logic
extended with generalized quantifiers, the standard Tarski style
semantics and the game-theoretic semantics are equivalent.
For the sake of brevity, we shall not attempt to formulate
extensions of Theorems \ref{thirdtheorem}
and \ref{fourththeorem} that apply to
extensions of $\mathcal{L}$ with quantifiers and
Turing machines with corresponding oracles.
Instead, further investigations in this direction are
left for the future.
%

%
%

%
\section{Concluding remarks}
It is easy to see that various interesting operators can be added to $\mathcal{L}$ without
sacrificing Turing-complete- ness. For example, second-order quantifiers can
easily be added. There are only finitely many ways to interpret a quantified
second-order variable in a finite model, and therefore K\"{o}nig's lemma
can still be applied so that Theorems \ref{secondtheorem} and \ref{fourththeorem}
hold. Also, it is possible to add to $\mathcal{L}$ an operator that, say, adds $|\mathcal{P}(W)|$
fresh elements to the domain $W$, and then extends the interpretations of selected 
relation symbols and second-order variables non-deterministically to all of the new domain.
%
%
In the finite, this operator does not add anything to 
the expressivity of $\mathcal{L}$, but of course more
delicate features of the underlying logic change.
Connections between $\mathcal{L}$ and team semantics
ought to be investigated thoroughly. Both P and NP can
be characterized nicely by logics based on team semantics; NP is
captured by both dependence logic and
IF logic, and P is captured on ordered models by 
\emph{inclusion logic} (see \cite{hella}).
Further interesting complexity classes will probably be characterized in terms of
logics based on team semantics in the near future.
We conjecture that by attaching suitable operators to the atoms of $\mathcal{L}$ of the type $k\in\mathbb{N}$,
it should be possible to extend $\mathcal{L}$ such that
resulting logics accomodate typical logics based on team semantics as fragments 
in a natural way. The game-theoretic approaches to team semantics developed in 
\cite{bradfield, gra, kuusisto, mann, vaananen} provide some starting points for related investigations.
Let $R$ be a binary relation symbol.
Let $\mathcal{L}_0$ denote the fragment of $\mathcal{L}$
that extends first-order logic by operators that enable the
the manipulation of the relation $R$ (only), the insertion
of fresh points to the domain, and recursive looping.
We conjecture that on models whose vocabulary 
contains the binary relation symbol $R$, already $\mathcal{L}_0$ is
Turing-complete. Indeed, this does not seem to be difficult
to prove using suitable gadgets, but we  leave it as a
conjecture at this stage.
Finally, it would be interesting to classify fragments of $\mathcal{L}$
according to whether their (finite) satisfiability problem is 
decidable. This would nicely extend the research on decidability
of fragments of first-order logic.

\nocite{*}
\bibliographystyle{eptcs}
\bibliography{generic}
\end{document}